%% file: PhD.tex
\documentclass[12pt,a4paper]{amsart}
\usepackage{amscd,amsmath,amssymb,color,graphicx}

\title
{Taut foliations\\
and the actions of fundamental groups\\
on leaf spaces and universal circles}
\author{Yosuke Kano}
\address{
Department of Mathematics and Informatics,
Graduate School of Science,
Chiba University,
Chiba 263-8522,
Japan
}
\email{ykano@g.math.s.chiba-u.ac.jp}
\thanks{
This work is partially supported by AGSST support 
program
for young researchers, 
2011, Chiba University, Japan
}
\newtheorem{thm}{Theorem}[section]
\newtheorem{lem}[thm]{Lemma}
\newtheorem{cor}[thm]{Corollary}
\newtheorem{prop}[thm]{Proposition}

\theoremstyle{definition}
\newtheorem{df}[thm]{Definition}

\theoremstyle{remark}
\newtheorem{rem}[thm]{Remark}
\newtheorem{ex}[thm]{Example}

\newcommand{\univ}{\operatorname{univ}}
\newcommand{\stab}{\operatorname{Stab}}
\newcommand{\homeo}{\operatorname{Homeo}^{+}}

\begin{document}
\maketitle
\begin{abstract}
Let $\mathcal{F}$ be a leafwise hyperbolic taut foliation 
of a closed $3$-manifold $M$
and let $L$ be the leaf space of the pullback of $\mathcal{F}$
to the universal cover of $M$.
We show that
if $\mathcal{F}$ has branching,
then the natural action of $\pi_1(M)$ on $L$ is faithful.
We also show that
if $\mathcal{F}$ has a finite branch locus $B$ 
whose stabilizer $\stab(B)$ acts on $B$ nontrivially,
then $\stab(B)$ is an infinite cyclic group
generated by an indivisible element of $\pi_1(M)$.
\end{abstract}


\section{Introduction}

Unless otherwise specified,
we assume throughout this article that
$M$ is a closed oriented 3-manifold
and $\mathcal{F}$ a codimension one transversely oriented,
leafwise hyperbolic, 
taut foliation of $M$.
Here we say that
$\mathcal{F}$ is {\it leafwise hyperbolic}
if there is a Riemannian metric on $M$ such that
the induced metrics on leaves have constant curvature $-1$,
and that $\mathcal{F}$ is {\it taut}
if there is a loop in $M$ which intersects every leaf of $\mathcal{F}$ 
transversely.
Note that by Candel \cite{can},
if $M$ is irreducible and atoroidal,
then every taut foliation of $M$ is leafwise hyperbolic.
Leafwise hyperbolic taut foliations
have been extensively investigated by many people in connection
with the theory of $3$-manifolds
(See e.g. Calegari's book \cite{c07}).
One of the most powerful methods of analyzing the structure of such foliations
is to consider canonical actions of $\pi_1(M)$ on $1$-manifolds
naturally associated with $\mathcal{F}$.
Two kinds of such $1$-manifolds are known.
The first one, denoted $L$, is 
the leaf space of $\widetilde{\mathcal{F}}$,
where $\widetilde{\mathcal{F}}$ is the pullback of $\mathcal{F}$ 
to the universal cover $\widetilde{M}$ of $M$.
The action of $\pi _1(M)$ on $\widetilde{M}$ 
induces an action of $\pi _1(M)$ on $L$.
In the sequel we refer to it as the {\it natural action}.
The second one is a universal circle.
Given $\mathcal{F}$,
by unifying circles at infinity of all the leaves of $\widetilde{\mathcal{F}}$,
Thurston \cite{thu} (See also Calegari-Dunfield \cite{cd03}) constructs
a universal circle with a canonical $\pi_1(M)$ action.

We say that $\mathcal{F}$ has {\it branching} if $L$ is non-Hausdorff.
The first result of this article is the following:

\vspace{9pt}

\noindent
{\bf Theorem 3.2.}\
{\it
If $\mathcal{F}$ has branching,
then the natural action on $L$ is faithful.
}

\vspace{9pt}

This result is obtained from an investigation of both actions of $\pi_1(M)$
on the leaf space and on the universal circle (see \S 3).
Notice that the hypothesis that $\mathcal{F}$ has branching is indispensable.
In fact, just consider a surface bundle over $S^1$ foliated by fibers.
Notice also that
Calegari and Dunfield \cite[Theorem 7.10]{cd03} have already shown that
for any taut foliation one can modify it
by suitable Denjoy-like insertions
so that the natural action associated with the resulting foliation becomes faithful.
In the case where the foliation is leafwise hyperbolic and has branching,
our result is stronger than theirs 
in that we assure faithfulness without performing any modifications.

Next we consider the stabilizer of a branch locus of $\mathcal{F}$.
We call a subset $B$ of $L$ a {\it positive {\rm (resp.} negative{\rm )} branch locus}
if $B$ contains at least two points
and if it can be expressed in the form
$B=\lim_{t\to 0}\nu_t$
for some interval
$\{\nu_t\in L\mid 0<t<\epsilon\}$ embedded in $L$
in such a way that the parameter $t$ is incompatible (resp. compatible)
with the orientation of $L$
(Note that $L$ has a natural orientation induced from the transverse orientation of $\widetilde{\mathcal{F}}$).
Branch loci have been studied e.g. in  \cite{f98}, \cite{she}.
For a branch locus $B$
we define the {\it stabilizer} of $B$ by
$\stab(B)=\{\alpha\in\pi_1(M)\mid\alpha(B)=B\}$.

In the case where a branch locus $B$ is finite,
we obtain the following results
about the action of $\stab(B)$ on $B$ (see \S 5 for details). 

\vspace{9pt}

\noindent
{\bf Theorem 5.2.}\
{\it
Let $B$ be a finite branch locus of $L$.
If an element of $\stab(B)$ fixes some point of $B$
then it fixes all the points of $B$.
}

\vspace{9pt}

We remark that for Anosov foliations Fenley \cite[Theorem D]{f98}
proves related results to this theorem.

Let $\pi:\widetilde M\to M$ be the covering projection.
For a leaf $\lambda$ of $\widetilde{\mathcal{F}}$,
we denote by $\underline{\lambda}$ the projected leaf 
$\pi(\lambda)$ of $\mathcal{F}$.

\vspace{9pt}

\noindent
{\bf Theorem 5.3.}\
{\it
Let $B$ be a finite branch locus of $L$.
Then, for any $\lambda\in B$,
\begin{enumerate}
\item[1.]
if $\stab(B)$ is trivial,
$\underline{\lambda}$ is diffeomorphic to a plane, and
\item[2.]
if $\stab(B)$ is nontrivial,
$\underline{\lambda}$ is diffeomorphic to a cylinder.
\end{enumerate}
}

\noindent
{\bf Theorem 5.6.}\
{\it
Let $B$ be a finite branch locus of $L$
with nontrivial stabilizer.
Then the stabilizer $\stab (B)$ is isomorphic to $\mathbb{Z}$.
}

\vspace{9pt}

We say that $\alpha\in\pi_1(M)$ is {\it divisible}
if there is some $\beta\in\pi_1(M)$ and an integer $k\geq 2$
such that $\alpha =\beta ^k$.
Otherwise we say $\alpha$ is {\it indivisible}.

\vspace{9pt}

\noindent
{\bf Theorem 5.7.}\
{\it
Let $B$ be a finite branch locus of $L$
such that $\stab (B)$ acts on $B$ nontrivially.
Then a generator of $\stab(B)$ $(\cong\mathbb{Z})$ is indivisible.
}

\vspace{9pt}

For an oriented loop $\gamma$ in $M$,
we say that $\gamma$ is {\it tangentiable}
if $\gamma$ is freely homotopic to a leaf loop 
(a loop contained in a single leaf) of $\mathcal{F}$,
and that $\gamma$ is positively (resp. negatively) {\it transversable}
if $\gamma$ is freely homotopic to a loop positively (resp. negatively) 
transverse to $\mathcal{F}$.
As a final topic of this article,
we study relations between the infiniteness of branch loci
and the existence of a non-transversable leaf loop in $M$ (see \S 6).
One of the results we obtain is the following:

\vspace{9pt}

\noindent
{\bf Theorem 6.5.}\
{\it
Suppose $\mathcal{F}$ has branching.
If there is a non-contractible leaf loop in $M$
which is not freely homotopic to a loop transverse to $\mathcal{F}$,
then $\mathcal{F}$ has an infinite branch locus.
}

\vspace{9pt}

This article is organized as follows.
In Section 2, we briefly review the Calegari-Dunfield construction of a 
universal circle.
Using their construction,
we prove the faithfulness of the natural action
of $\pi_1(M)$ on $L$ in Section 3.
In Section 4, we introduce a notion of comparable sets
and give several basic properties of such sets,
which are applied in Section 5 to
the investigation of the structure of finite branch loci and their stabilizers.
In Section 6, 
we study how the non-transversability
of leaf loops in $M$ is related 
to the infiniteness of branch loci in $L$.


\section{Universal circle}
In this section we briefly recall the definition  of a universal circle after \cite{cd03}.

Let $M$, $\mathcal{F}$ and $L$ be as in the introduction.
For $\lambda, \mu\in L$
we write $\lambda <\mu$
if there is an oriented path in $\widetilde{M}$
from $\lambda$ to $\mu$
which is positively transverse to $\widetilde{\mathcal{F}}$.
We say that $\lambda$ and $\mu$ are {\it comparable}
if either $\lambda\le\mu$ or $\lambda\ge\mu$.
For a leaf $\lambda$ of $\widetilde{\mathcal{F}}$,
the {\it endpoint map} $e:T_p\lambda-\{0\}\to S^1_\infty(\lambda)$
from the tangent space of $\lambda$ at $p$
to the ideal boundary of $\lambda$
takes a vector $v$
to the endpoint at infinity of the geodesic ray $\gamma$
with $\gamma(0)=p$ and  $\gamma '(0)=v$.
The {\it circle bundle at infinity} is the disjoint union
$E_\infty =\bigcup _{\lambda\in L}S^1_\infty (\lambda)$
with the finest topology such that the endpoint map
$e:T\widetilde{\mathcal{F}}-(\text{zero section})\to E_\infty$
is continuous.
A continuous map $\phi :X\to Y$
between oriented $1$-manifolds homeomorphic to $S^1$
is {\it monotone}
if it is of mapping degree one
and if the preimage of any point of $Y$ is contractible.
A {\it gap} of $\phi$ is the interior in $X$ of such a preimage.
The {\it core} of $\phi$ is the complement of the union of gaps.

\begin{df}
A {\it universal circle} $S^1_{\univ}$ for $\mathcal{F}$
is a circle together with a homomorphism
$\rho_{\univ}:\pi_1 (M)\to\homeo(S^1_{\univ})$
and a family of monotone maps
$\phi _{\lambda}:S^1_{\univ}\to S^1_\infty (\lambda)$
, $\lambda\in\widetilde{\mathcal{F}}$,
satisfying the following conditions:
\begin{enumerate}
\item[1.]
For every $\alpha\in\pi_1 (M)$, the following diagram commutes:
$$
\begin{CD}
S^1_{\univ}@>\rho _{\univ}(\alpha)>>S^1_{\univ}\\
@V\phi_ {\lambda}VV@V\phi _{\alpha(\lambda)}VV\\
S^1_\infty(\lambda)@>\alpha>>S^1_\infty(\alpha (\lambda))\\
\end{CD}
$$
\item[2.]
If $\lambda$ and $\mu$ are incomparable
then the core of $\phi _{\lambda}$ is contained
in the closure of a single gap of $\phi_{\mu}$
and {\it vice versa}.
\end{enumerate}
\end{df}

Calegari-Dunfield's way of construction
of a universal circle is as follows.
Let $I=[0,1]$ be the unit interval.
A {\it marker} for $\mathcal{F}$
is a continuous map $m:I\times\mathbb{R}^{+}\to\widetilde{M}$
with the following properties:
(1) For each $s\in I$, the image $m(s\times\mathbb{R}^{+})$ is
a geodesic ray in a leaf of $\widetilde{\mathcal{F}}$.
We call these rays the {\it horizontal rays} of $m$.
(2) For each $t\in\mathbb{R}^{+}$,
the image $m(I\times t)$ is transverse to
$\widetilde{\mathcal{F}}$
and
of length smaller than some constant depending only 
on $\widetilde{\mathcal{F}}$.

We use the interval notation $[\lambda ,\mu]$
to represent the oriented image of an injective continuous map
$c:I\to L$ such that $c(0)=\lambda$ and $c(1)=\mu$,
and refer it to the {\it interval} from $\lambda$ to $\mu$.
Here, notice that the orientation of such an interval is induced from that of $I$
(not from that of $L$).

Let  $J=[\lambda ,\mu]$ be an interval in $L$
and let $m$ be a marker which intersects only leaves of $\widetilde{\mathcal{F}}|_J$.
Then the endpoints of the horizontal rays of $m$ form an interval
in $E_\infty |_J$
which is transverse to the circle fibers.
By abuse of notation we refer to
such an interval as a marker.

For each $\nu\in J$,
the intersection of $S^1_\infty (\nu)$ with the union of all markers 
are dense in $S^1 _\infty (\nu)$.
If two markers $m_1,m_2$ in $E_\infty |_J$ are not disjoint,
their union $m_1\cup m_2$ is also an interval transverse to the circle fibers.
It follows that
a maximal such union of markers
is still an interval.
Again by abuse of notation we call such an interval a marker.

A continuous section $\tau :J\to E_\infty |_J$ is {\it admissible}
if the image of $\tau$ does not cross (but might run into) any marker.
The {\it leftmost section}
$\tau(p,J):J\to E_\infty |_J$
starting at $p\in S^1_\infty (\lambda)$
is an admissible section
which is anticlockwisemost among all such sections
if the order of $J$ is compatible with that of $L$,
and clockwisemost otherwise.
For any $p$ the leftmost section starting at $p$ exists.

Let $B=\lim _{t\to 0}\nu _t$ be a branch locus
and let $\mu _1, \mu _2\in B$.
For each $t>0$,
let $\alpha _t=[\mu_1 ,\nu _t]$
and $\beta _t=[\nu _t,\mu_2]$.
Then, we can define a map
$r_t: S^1_\infty (\mu_1)\to S^1_\infty (\mu_2)$
by $r_t(p)=\tau(\tau(p,\alpha_t)(\nu _t),\beta_t)(\mu _2)$.
As $t$ tends to $0$, $r_t$ converges to a constant map.
We denote the image of the constant map 
by $r(\mu _1,\mu _2)\in S^1_\infty (\mu _2)$
and call it
the {\it turning point from $\mu _1$ to $\mu _2$}.

Given an incomparable pair $\lambda ,\mu\in L$,
we define a \lq\lq {\it path}" from $\lambda$ to $\mu$ 
to be a disjoint union of finitely many intervals
$[\hat{\nu}_{i-1}, \check{\nu}_{i}]$, $1\le i\le n$, in $L$
(some of them may be degenerated to singletons),
with the following properties:
(1) $\hat{\nu}_0=\lambda$ and $\check{\nu}_{n}=\mu$,
(2) $\check{\nu}_i$ and $\hat{\nu}_i$ belong to
a common branch locus for each $1\le i\le n-1$, and
(3) $n$ is minimal under the conditions (1) and (2).
Note that
a path connecting any two points in $L$ exists and is unique.

For a point $p$ in $S^1_\infty (\lambda)$,
the {\it special section}
$\sigma _p:L\to E_\infty$ at $p$
is defined as follows.
First, set $\sigma _p(\lambda)=p$.
Next, pick any point $\mu\in L$.
We define $\sigma _p(\mu)$ as follows:
When $\mu$ is comparable with $\lambda$,
then $\sigma _p$ is defined on $[\lambda,\mu]$ to be
the leftmost section starting at $p$.
When $\mu$ is incomparable with $\lambda$,
let
$\coprod _{i=1}^n[\hat{\nu}_{i-1},\check{\nu}_{i}]$ $(n>1)$
be the path from $\lambda$ to $\mu$.
We then put
$r=r(\check{\nu}_{n-1},\hat{\nu}_{n-1})
\in S^1_\infty (\hat{\nu}_{n-1})$
and define $\sigma _p$ on the interval $[\hat{\nu}_{n-1},\check{\nu}_{n}]$
by $\sigma _p=\sigma _r$.
This completes the definition of $\sigma_p$.

Let $\mathfrak{S}$ be the union of the special sections $\sigma _p$
as $p$ varies over all points
in all circles $S^1_\infty (\lambda)$
of points $\lambda$ in $L$.
By \cite[Lemma 6.25]{cd03},
the set $\mathfrak{S}$ admits a natural circular order.
The universal circle $S^1_{\univ}$ will be derived from $\mathfrak{S}$
as a quotient of the order completion of $\mathfrak{S}$
with respect to the circular order.
Then one can observe
that any element of $S^1_{\univ}$ is a section $L\to E_\infty$.


\section{Faithfullness of the action}

In this section,
we show that if $\mathcal{F}$ has branching,
then the natural action of $\pi_1(M)$
on the leaf space $L$ is faithful.

As explained in {\S}2,
every element $\sigma$ of $S^1_{\univ}$ is a section
$\sigma :L\to E_\infty =\bigcup _{\lambda\in L} S^1_\infty (\lambda)$
and that the maps
$\phi _\lambda :S^1_{\univ}\to S^1_\infty (\lambda)$
are defined by
$\phi_\lambda(\sigma)=\sigma(\lambda)$.
For a point $x$ in $S^1_\infty (\lambda)$,
we define a (possibly degenerate) closed interval $I_x$
in $S^1_{\univ}$ by
$I_x=\{\sigma\in S^1_{\univ}\mid\sigma (\lambda)=x\}$.
Then, for any $x$ the interval $I_x$ is nonempty
because the special section $\sigma _x$ at $x$
belongs to $I_x$.

Let $\lambda\in L$ and $\alpha\in\pi _1(M)$ be
such that $\alpha(\lambda)=\lambda$.
Then $\alpha$,
as the restriction of a covering transformation of $\widetilde M$ to $\lambda$, induces an isometry
 of the hyperbolic plane $\lambda$,
(hence also a projective transformation of $S^1_\infty (\lambda)$).
We notice that this isometry is a hyperbolic element
(meaning that its trace is greater than $2$).
In fact,
since it has no fixed points in $\lambda$,
it is not elliptic.
If it were parabolic, then it would yield in $M$
a non-contractible loop whose length can be made arbitrarily small,
contradicting the compactness of $M$.

The following is a key lemma.

\begin{lem}
\label{key}
Let $B=\lim _{t\to 0}\nu_t$ be
a branch locus of the leaf space of $L$.
If $\alpha\in\pi _1(M)$ fixes
two distinct points $\mu _1$ and $\mu _2$ in $B$
and also fixes the interval $\{\nu _t\mid 0<t<\epsilon\}$ pointwise,
then $\alpha$ is trivial in $\pi _1(M)$.
\end{lem}
\begin{proof}
Suppose $\alpha$ is nontrivial.
Let $p_1,q_1\in S^1_\infty (\mu _1)$
and $p_2,q_2\in S^1_\infty (\mu _2)$
be the fixed points of $\alpha$,
and let $r_1\in S^1_\infty (\mu _1)$
be the turning point from $\mu _2$ to $\mu _1$.
Without loss of generality,
we assume that $p _1\ne r_1$.
Note that by construction of the universal circle,
the special sections
$\sigma _{p_i}$ and $\sigma _{q_i}$ in $S^1_{\univ}$
are fixed by $\rho _{\univ}(\alpha)$ for $i=1,2$,
and therefore the images
$\phi _{\nu _t}(\sigma _{p_i})$ and $\phi _{\nu _t}(\sigma _{q_i})$
are fixed by $\alpha$ for any $t\in (0,\epsilon )$.

We claim that 
if $t$ is sufficiently close to $0$, then 
$\phi _{\nu _t}(\sigma _{p_1})$ and $\phi _{\nu _t}(I_{r_1})$
are disjoint in $S^1_\infty (\nu _t)$.
Take two distinct points $x$ and $y$
in $S^1_\infty (\mu _1)-\{p_1,r_1\}$ so that
the $4$-tuple $(p_1,x,r_1,y)$ 
lies in circular order.
Then, for sufficiently small $t>0$,
the $4$-tuple 
$(\sigma _{p_1}(\nu _t),\sigma _{x}(\nu _t),
\sigma _{r_1}(\nu _t),\sigma _{y}(\nu _t))$
lies in $S^1_\infty (\nu _t)$ 
also in circular order.
Let $K$ be the closed interval in $S^1_\infty (\nu _t)$
with boundary points $\sigma _{x}(\nu _t)$ and $\sigma _{y}(\nu _t)$
which contains $\sigma _{r_1}(\nu _t)$.
Since $I_{r_1}$ contains $\sigma _{r_1}$
but does not contain $\sigma _{p_1},\sigma _{x}$ and $\sigma _{y}$,
it follows that
$\phi _{\nu _t}(I_r)$ is contained in $K$.
In particular, 
$\phi _{\nu _t}(\sigma _{p_1})$ and $\phi _{\nu _t}(I_{r_1})$
are disjoint.
This shows the claim.

For $t$ sufficiently close to $0$,
the two points $\sigma _{p_2}(\nu _t)$ and $\sigma _{q_2}(\nu _t)$
are distinct.
Since $\sigma _{p_2}$ and $\sigma _{q_2}$ are contained in $I_{r_1}$,
the $3$ points
$\sigma _{p_1}(\nu _t),\sigma _{p_2}(\nu _t)$ and $\sigma _{q_2}(\nu _t)$
are also mutually distinct.
Thus,
we find at least $3$ fixed points of $\alpha$ in $S^1_\infty (\nu _t)$,
contradicting to the fact that
$\alpha$ is a nontrivial orientation preserving isometry
of the hyperbolic plane $\nu _t$.
\end{proof}

Now, the first main result of this article is the following:

\begin{thm}
\label{main}
Let $M$ be a closed oriented $3$-manifold,
and $\mathcal{F}$ a transversely oriented
leafwise hyperbolic taut foliation of $M$.
If $\mathcal{F}$ has branching,
then the natural action of $\pi _1(M)$
on the leaf space of $\widetilde{\mathcal{F}}$ is faithful.
\end{thm}
\begin{proof}
This is a direct consequence of Lemma \ref{key}.
\end{proof}


\section{Comparable sets}

Contents of this section are unrelated to leafwise hyperbolicity.
For $\alpha\in \pi _1(M)$,
we define the {\it comparable set}  $C_\alpha$ for $\alpha$
to be the subset of $L$ consisting of points $\lambda$
such that $\lambda$ and $\alpha(\lambda)$ are comparable.
Below we  collect some basic properties of comparable sets.

Obviously,
$\alpha (C_\alpha)=C_\alpha$,
$C_\alpha =C_{\alpha^{-1}}$ and
$C_\alpha\subset C_{\alpha ^k}$ for every $k>0$.

We say that $\mathcal{F}$ has
{\it one-sided branching in the positive
{\rm (resp.} negative{\rm )} direction}
if $L$ has positive (resp. negative) branch loci
but has no negative (resp. positive) ones.
If $L$ has both positive loci and negative loci,
then we say $\mathcal{F}$ has {\it two-sided branching}.

\begin{lem}
\label{c1}
Let $\mathcal{F}$ have one-sided branching in the positive direction,
and let $\alpha\in\pi _1(M)$.
Suppose $\lambda$ and $\mu$ are points in $L$ such that
$\lambda$ is a common lower bound of $\mu$ and $\alpha(\mu)$.
Then $\lambda\in C_\alpha$.
\end{lem}
\begin{proof}
Since the natural action preserves the order of $L$,
the inequality $\lambda <\mu$ implies $\alpha (\lambda)<\alpha (\mu)$.
Thus, by the hypothesis,
$\alpha(\mu)$ is a common upper bound of $\lambda$ and $\alpha (\lambda)$.
Since $\mathcal{F}$ has no branching in the negative direction,
it follows that $\lambda$ and $\alpha(\lambda)$ are comparable.
\end{proof}

From this lemma we see the following fact:
Let $\mathcal{F}$ and $\alpha$ be as above.
Then, there is $\lambda\in L$ such that
$\{\mu\in L\mid\mu <\lambda\}\subset C_\alpha$.

\begin{lem}
\label{c2}
Let $\alpha\in\pi _1(M)$ and let $\lambda,\mu\in C_\alpha$.
Then the path $\gamma$ from $\lambda$ to $\mu$
is entirely contained in $C_\alpha$.
Furthermore,
if $\gamma$ is represented as
$\gamma =\coprod _{i=1}^n[\hat{\nu}_{i-1},\check{\nu}_{i}]$
$(\hat{\nu}_0=\lambda$, $\check{\nu}_{n}=\mu )$
by using a union of intervals,
then $\check{\nu}_i,\hat{\nu}_i$ are fixed by $\alpha$
for each $1\le i\le n-1$.
\end{lem}
\begin{proof}
Let $p:L\to K$ be the Hausdorffification of $L$
(namely, $K$ is the space obtained from $L$
by identifying each of the branch loci to a single point).
Put $\delta =p(\gamma)$.
We see that $\delta$ is
a {\lq\lq}real path{\rq\rq} zigzaggedly embedded in $K$.
Let $a$ and $b$ be the embedded intervals in $K$
joining respectively
$p(\lambda)$ to $p(\alpha(\lambda))$
and $p(\mu)$ to $p(\alpha(\mu))$. 
Then,
$\delta *b*\alpha (\delta)^{-1}*a^{-1}$
is a loop in $K$.
Here, we notice that some of the intervals
$a$, $b$ and $[\hat{\nu}_{i-1},\check{\nu}_{i}]$
may be degenerate.
We first consider the case where all the intervals are non-degenerate.
We may assume without loss of generality that $\lambda <\check{\nu}_1$.
We may also assume that $\lambda <\alpha(\lambda)$,
because otherwise we have $\lambda <\alpha ^{-1}(\lambda)$
and so we may consider $\alpha ^{-1}$ instead of $\alpha$.
Now, we have four cases:
(1) $\hat{\nu}_{n-1}<\mu$ and $\alpha(\mu)<\mu$,
(2) $\hat{\nu}_{n-1}>\mu$ and $\alpha(\mu)>\mu$,
(3) $\hat{\nu}_{n-1}<\mu$ and $\alpha(\mu)>\mu$, and
(4) $\hat{\nu}_{n-1}>\mu$ and $\alpha(\mu)<\mu$.
In cases (1) and (2),
we see that
$\delta$ and $a*\alpha (\delta)*b^{-1}$
are both embedded paths from $p(\lambda)$ to $p(\mu)$.
Since $K$ is $1$-connected,
these two paths must coincide with each other.
Furthermore,
since the break points of these two paths must coincide
with each other in order-preserving manner,
we have that
$\alpha$ fixes $p(\check{\nu}_i)$ for all $1\le i\le n-1$.
By pulling back this information to $L$ by $p$
and taking account of the definition of the path,
we finally have that
$\alpha$ fixes $\check{\nu}_i$ and $\hat{\nu}_i$ for all $1\le i\le n-1$.
In cases (3) and (4),
we have two embedded paths joining $p(\lambda)$ to $p(\alpha(\mu))$.
A similar argument as above shows the same conclusion.
This shows the lemma in the case
where all the intervals involved are non-degenerate.
In the case where some of the intervals are degenerate,
by regarding degenerate intervals
as {\lq\lq}infinitesimally non-degenerate{\rq\rq} intervals,
we can treat this case essentially in the same way as in the non-degenerate case.
We would like to leave the rigorous description to the reader.
This proves the lemma.
\end{proof}

\begin{prop}
\label{co}
For any $\alpha\in\pi_1(M)$,
$C_\alpha$ is connected and open.
\end{prop}
\begin{proof}
Connectedness is already shown in Lemma \ref{c2}.
We will prove openness.
Let $\lambda$ be any point in $C_\alpha$.
If $\alpha (\lambda)\ne\lambda$
then the open interval bounded by
$\alpha ^{-1}(\lambda)$ and $\alpha (\lambda)$
is contained in $C_\alpha$ and contains $\lambda$.
Thus, $\lambda$ is an interior point of $C_\alpha$.
Next, we consider the case where $\alpha (\lambda)=\lambda$.
Take any point $\mu\in L$ with $\lambda<\mu$.
Then the interval $[\lambda ,\mu]$ is mapped by $\alpha$ 
orientation preservingly onto
the interval $[\lambda ,\alpha (\mu)]$.
Since $L$ is an oriented $1$-manifold,
there must exist $\nu\in (\lambda,\mu]$ such that
$[\lambda ,\nu)$ is contained in
$[\lambda ,\mu]\cap [\lambda ,\alpha (\mu)]$.
This implies that
$[\lambda ,\nu )$ is contained in $C_\alpha$.
Similarly, we can find $\eta <\lambda$ such that
$(\eta ,\lambda ]$ is contained in $C_\alpha$.
Consequently,
we have $\lambda\in (\eta ,\nu )\subset C_\alpha$,
which means $\lambda$ is an interior point of $C_\alpha$.
This proves the proposition.
\end{proof}

For a path
$\gamma =\coprod _{i=1}^n[\hat{\nu}_{i-1},\check{\nu}_{i}]$,
we call $n$ the {\it length} of $\gamma$
and denote it by $l(\gamma )$.

\begin{prop}
\label{cempty}
Let $\alpha\in\pi _1(M)$ and $\lambda\in L$
be such that $\lambda\notin C_\alpha$.
If the path from $\lambda$ to $\alpha(\lambda)$ 
has odd length,
then $C_{\alpha ^k}=\emptyset$ for any $k>0$.
\end{prop}
\begin{proof}
Let $\gamma$ be the path joining $\lambda$ to $\alpha(\lambda)$.
Since $l(\gamma)$ is odd,
there occur no overlappings in composing $k$ paths
$\gamma ,\alpha (\gamma ),\cdots ,\alpha ^{k-1}(\gamma )$
successively, and the result
$\gamma *\alpha (\gamma )*\cdots *\alpha ^{k-1}(\gamma )$
is the path from
$\lambda$ to $\alpha ^k (\lambda)$.
It follows that
$\lambda$ and $\alpha ^k (\lambda)$ are incomparable.
Put $A=\coprod _{i\in\mathbb{Z}}\alpha ^i (\gamma)$.
Notice that if $\mu\in A$ then $\mu\notin C_\alpha$,
so we assume $\mu\notin A$.
For two points $\nu$ and $\eta$ in $L$
we denote by $\gamma ^\nu _\eta$
the path from $\nu$ to $\eta$.
Let $\delta =\bigcap _{\nu\in A}\gamma _\nu ^\mu$.
Then, $\delta$ is disjoint from $A$.
But, there exists a branch locus $B$ intersecting $A$
such that either
$B\cap\delta\ne\emptyset$, or
one of the endpoints of $\delta$ converges to $B$.
Thus,
we have an estimate
\[
\min\{
l(\gamma ^\nu _\eta)\mid
\nu\in B,\eta\in\alpha ^k(B),k>0
\}\ge 2
\]
because $l(\gamma )\ge 3$.
Since the path $\gamma ^\mu _{\alpha ^k(\mu )}$
must pass through the branch loci
$B$ and $\alpha ^k (B)$,
we have $l(\gamma ^\mu _{\alpha ^k(\mu )})\ge 2$.
Namely, $\mu$ and $\alpha ^k (\mu)$ are incomparable.
\end{proof}

\begin{lem}
\label{return}
Let $\alpha\in\pi _1(M)$ and $\lambda\in L$
be such that $\lambda\notin C_\alpha$
and that $\lambda\in C_{\alpha ^k}$ for some $k>1$.
Let
$\gamma =\coprod _{i=1}^n[\hat{\nu}_{i-1},\check{\nu}_{i}]$
$(\hat{\nu}_0=\lambda$, $\check{\nu}_{n}=\alpha (\lambda ))$
be the path from $\lambda$ to $\alpha (\lambda )$. 
Then $\alpha (\check{\nu}_m)=\hat{\nu}_m$
and $\alpha ^k(\check{\nu}_m)=\check{\nu}_m$
where $m=l(\gamma)/2$
$($which is an integer by the above proposition$)$.
\end{lem}
\begin{proof}
Let $\gamma _j$ be the path from $\lambda$ to $\alpha ^j(\lambda )$,
and let $\delta _0$ and $\delta _1$
be the paths from
$\lambda$ to $\check{\nu}_m$, and
from
$\hat{\nu}_m$ to $\alpha (\lambda)$, respectively.
By reversing the transverse orientation of $\mathcal{F}$ if necessary,
we can assume that $\check{\nu}_m$ and $\hat{\nu}_{m}$
belong to a common positive branch locus.

\begin{figure}[btp]
\begin{center}
\input{rfig.tex}
\caption{
$\alpha(\delta_0)$ is shown as a broken line
in the case $\check{\nu}_m\in C_\alpha$,
and as a dotted line
in the case $\alpha (\check{\nu}_m)\in (\hat{\nu}_m,\check{\nu}_{m+1}]$.
}
\label{rfig}
\end{center}
\end{figure}
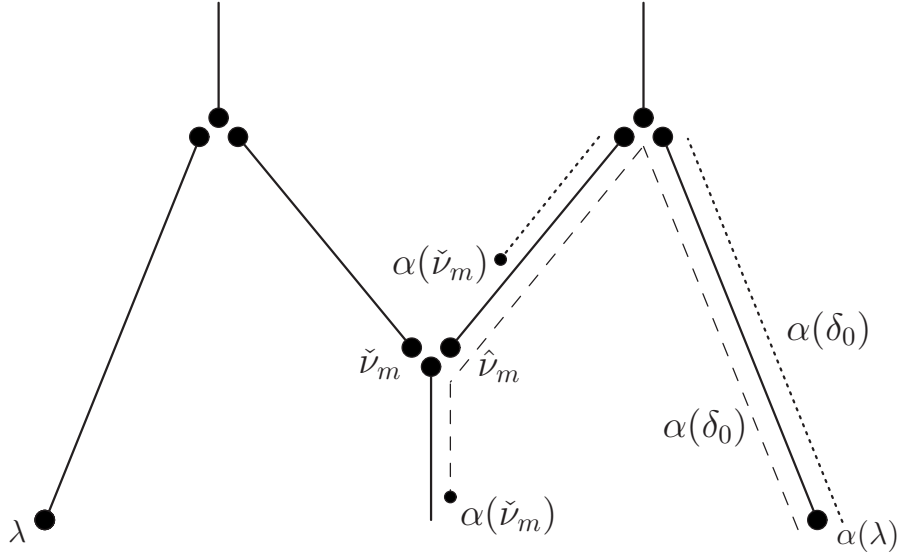

First, we show that $\check{\nu}_m\notin C_\alpha$.
Suppose on the contrary that $\check{\nu}_m\in C_\alpha$.
Note that the length of the path $\alpha (\delta _0)$
joining $\alpha (\lambda)$ to $\alpha (\check{\nu}_m)$ is $l(\gamma)/2$.
So if $\check{\nu}_m$ and $\alpha (\check{\nu}_m)$ are comparable,
the intersection $\gamma\cap\alpha (\delta _0)$
must coincide with $\delta _1$ as a set.
Therefore $\check{\nu}_m >\alpha (\check{\nu}_m)$.
See Fig. \ref{rfig}.
Then $\check{\nu}_m >\alpha ^{k-1}(\check{\nu}_m)$,
and we have
$\gamma _k=\delta _0*
[\check{\nu}_{m},\alpha ^{k-1}(\check{\nu}_m)]*
\alpha ^{k-1}(\delta _1)$.
Since $\gamma _k$ passes through
$\alpha ^{k-1}(\check{\nu}_m)$ and $\alpha ^{k-1}(\hat{\nu}_m)$,
it follows that $\lambda$ and $\alpha ^k(\lambda )$ are incomparable,
which contradicts the choice of $\lambda$.

Next, we show that
$\alpha (\check{\nu}_m)\notin (\hat{\nu}_m,\check{\nu}_{m+1}]$.
Suppose not.
Then the branch locus obtained from the embedded interval
$(\hat{\nu}_m,\alpha (\check{\nu}_m))$
contains $\alpha (\hat{\nu}_m)$, that is,
$\hat{\nu}_m$ and $\alpha (\hat{\nu}_m)$ are comparable.
See Fig. \ref{rfig}.
Since we are assuming 
 that $\check{\nu}_m$ and $\hat{\nu}_{m}$
belong to a common positive branch locus,
we have $\hat{\nu}_m<\alpha (\hat{\nu}_m)$.
Then $\hat{\nu}_m <\alpha ^{k-1}(\hat{\nu}_m)$,
and therefore
$\gamma _k=\delta _0*
[\hat{\nu}_{m},\alpha ^{k-1}(\hat{\nu}_m)]*
\alpha ^{k-1}(\delta _1)$.
Since $\gamma _k$ passes through
$\check{\nu}_m$ and $\hat{\nu}_m$,
it follows that
$\lambda$ and $\alpha ^k(\lambda )$ are incomparable,
which is a contradiction.

Finally,
we consider other cases.
If $\alpha (\check{\nu}_m)\ne\hat{\nu}_m$, 
we have
$l(\alpha ^{j+1}(\gamma)-\alpha^j (\gamma))>l(\gamma)/2$
for all $0\le j<k$.
Therefore, we have
\[
1<l(\gamma _1)<l(\gamma _2)<\cdots <l(\gamma _k)=1
\]
This contradiction shows that
$\alpha (\check{\nu}_m)=\hat{\nu}_m$.
We also have that $\alpha ^k(\check{\nu}_m)=\check{\nu}_m$.
Otherwise, $\gamma _k=\delta _0*\alpha ^k (\delta _0)^{-1}$,
and therefore $\gamma _k$ passes through
$\check{\nu}_m$ and $\alpha^k (\check{\nu}_m)$
which belong the common branch locus.
It follows that $\lambda\notin C_{\alpha ^k}$,
which is a contradiction.
\end{proof}

\begin{lem}
\label{invb}
Let $\alpha\in\pi _1(M)$
and let $B$ be an $\alpha$-invariant branch locus.
If $\{\nu_t\}_{0<t<\epsilon}$ be an embedded interval
such that $B=\lim _{t\to 0}\nu _t$,
then there exists $0<\epsilon'<\epsilon$ such that
$\nu _t$ is in $C_\alpha$ for any $t\in (0,\epsilon ')$.
\end{lem}
\begin{proof}
Let $\{\nu_t\}_{0<t<\epsilon}$ be as in the hypothesis of the lemma.
Then, $\alpha (B)=\lim _{t\to 0}\alpha (\nu _t)$.
Since $B=\alpha (B)$, 
two intervals
$\{\nu _t\}_{0<t<\epsilon }$ and
$\{\alpha (\nu _t)\} _{0<t<\epsilon }$
are both asymptotic to $B$
from the same direction as $t$ tends to $0$.
This with the fact that $L$ is a $1$-manifold implies that
the two intervals coincide near $B$.
Thus, the conclusion of the lemma follows.
\end{proof}


\section{branch loci and their stabilizers}

In this section we focus on a branch locus of the leaf space $L$.
We consider the case where a branch locus is a finite set
and clarify the structure of the stabilizer of such a locus.

\begin{lem}
\label{fixb}
Let $B$ be a finite branch locus
and let  $\alpha\in\stab(B)$.
If $\rho _{\univ}(\alpha )$ has a fixed point in $S^1_{\univ}$,
then $\alpha$ fixes $B$ pointwise.
\end{lem}
\begin{proof}
Let $\alpha\in\stab(B)$ be a nontrivial element
satisfying the hypothesis of the lemma,
and let $\lambda$ be any point of $B$.
Then, since $B$ is finite,
there exists some $k\in \mathbb{N}$ such that
$\alpha^k(\lambda)=\lambda$.
Notice here that $\alpha^k$ is nontrivial in $\pi_1(M)$,
because by tautness of $\mathcal{F}$ and by Novikov's theorem (\cite{nov})
our manifold $M$ is aspherical
and hence has no torsion in $\pi_1(M)$ (\cite[Corllary 9.9]{torsion}).
Now, let us suppose by contradiction that $\alpha(\lambda )\ne\lambda$.
Let $r\in S^1_\infty (\lambda)$ be the turning point
from $\alpha (\lambda)$ to $\lambda$
and let $p\in S^1_{\infty}(\lambda)$ be
one of the two fixed points of $\alpha^k$
which is different from $r$.
Then the special section $\sigma _p$ in $S^1_{\univ}$
is fixed by $\rho _{\univ}(\alpha ^k)$. 
This with the hypothesis that
$\rho _{\univ}(\alpha)$ has a fixed point implies that
$\sigma _p$ must be fixed by $\rho _{\univ}(\alpha )$ itself.
So we have $\rho _{\univ}(\alpha )(\sigma _p)\in I_p$.
On the other hand, since $\alpha (p)\in\alpha (\lambda )$,
it follows from the definition of turning point that
$\rho _{\univ}(\alpha )(\sigma _p)=\sigma _{\alpha (p)}\in I_r$.
This is a contradiction because $I_p$ and $I_r$ are disjoint.
\end{proof}

\begin{thm}
\label{fix}
Let $M$ be a closed oriented $3$-manifold,
and $\mathcal{F}$ a transversely oriented
leafwise hyperbolic taut foliation of $M$.
Suppose $\mathcal{F}$ has a finite branch locus $B$.
If an element of $\stab(B)$ fixes some point of $B$
then it fixes all the points of $B$.
\end{thm}
\begin{proof}
Let $\lambda$ be the $\alpha$-fixed point in $B$,
and let $p,q\in S^1_\infty (\lambda )$ be the fixed points of $\alpha$.
Then $\sigma _p,\sigma _q\in S^1_{\univ}$ are fixed by $\rho _{\univ}(\alpha)$.
The proof follows by Lemma \ref{fixb}.
\end{proof}

The next result gives information on topological types of leaves
in a finite branch locus.

\begin{thm}
\label{pc}
Let $M$ and $\mathcal{F}$ be as in Theorem \ref{fix}
and let $B$ be a finite branch locus of $L$.
Then, for any $\lambda\in B$,
\begin{enumerate}
\item[1.]
if $\stab(B)$ is trivial,
$\underline{\lambda}$ is diffeomorphic to a plane, and
\item[2.]
if $\stab(B)$ is nontrivial,
$\underline{\lambda}$ is diffeomorphic to a cylinder.
\end{enumerate}
\end{thm}
\begin{proof}
Let $\lambda\in B$.
Since $\mathcal{F}$ is taut,
by Novikov's theorem (\cite{nov})
the inclusion map of each leaf of $\mathcal{F}$
into $M$ is $\pi_1$-injective.
So,
if $\underline{\lambda}$ is not a plane,
there exists a nontrivial element $\alpha\in\pi_1(M)$
such that $\alpha(\lambda)=\lambda$.
This $\alpha$ must belong to $\stab (B)$, 
showing the first statement of the theorem.

To prove the second statement, suppose that $\stab (B)$ is nontrivial.
Then, we can first observe that
 $\underline{\lambda}$ is not a plane.
In fact, let $\gamma$ be any nontrivial element of  $\stab (B)$.
Since $B$ is finite,
$\gamma^n(\lambda)=\lambda$ for some $n\in \mathbb{N}$.
By the same argument as in the proof of Lemma \ref{fixb}
we see that $\gamma^n$ nontrivial in $\pi_1(M)$.
This shows the observation.
Now, by way of contradiction, let us assume $\underline{\lambda}$ is not a cylinder, either.
Then, again by $\pi_1$-injectivity of the inclusion $\lambda\to M$,
we can find elements $\alpha,\beta\in\pi _1(M)$ generating a free subgroup of rank $2$
such that
$\alpha(\lambda)=\beta(\lambda)=\lambda$.
These two elements are hyperbolic as isometries of
$\lambda$
and having no common fixed point on $S^1_\infty (\lambda)$.
Let $\mu$ be another leaf in $B$,
and let $r\in S^1_\infty (\lambda)$ be
the turning point 
from $\mu$ to $\lambda$.
By exchanging $\alpha$ and $\beta$ if necessary,
we may assume $\alpha(r)\ne r$.
Then,
$\alpha ^k (r)\ne\alpha ^l (r)$ for any $k\ne l\in\mathbb{Z}$.
Pick a point $s\in S^1_\infty (\mu)$
and consider the special section $\sigma _s$ at $s$.
Then,
$\rho _{\univ}(\alpha ^k)(\sigma _s)=\sigma _{\alpha ^k(s)}$
is the special section at $\alpha ^k(s)$.
Since
$\sigma _{\alpha ^k(s)}(\lambda)
=\phi _\lambda\circ\rho _{\univ}(\alpha ^k)(\sigma _s)
=\alpha ^k\circ\phi _\lambda (\sigma _s)=\alpha ^k(r)$,
it follows that 
$\alpha ^k(r)$ is the turning point from
$\alpha ^k(\mu)$ to $\lambda$.
In particular, $\alpha ^k(\mu)\ne\alpha ^l(\mu)$ for $k\ne l$,
hence,
$B$ contains infinitely many elements $\alpha ^k(\mu)$,
$k\in\mathbb{Z}$,
contradicting the finiteness of $B$.
\end{proof}

\begin{rem}
The author does not know
whether or not
there exists a finite branch locus which has a trivial stabilizer.
\end{rem}

\begin{prop}
\label{single}
Let $B=\{\lambda _1,\cdots ,\lambda _n\}$ be a finite branch locus
which has a nontrivial stabilizer
and let $r_i^j\in S^1_\infty (\lambda_i)$ be the turning point
from $\lambda _j$ to $\lambda _i$.
Then there exists $1\le k\le  n$ such that
the set of turning points $\{ r_k^j\mid j\ne k\}$
is a single point in $S^1_\infty (\lambda _k)$.
\end{prop}
\begin{proof}
By Theorem \ref{pc},
each $\underline{\lambda _i}$ is a cylindrical leaf.
Let $\gamma$ be an element in $\stab(B)$
such that $\gamma(\lambda_1)=\lambda_1$
and that a loop in $M$ representing $\gamma$ is freely homotopic to
a loop in $\underline{\lambda_1}$ generating $\pi_1(\underline{\lambda_1})$.
By Theorem \ref{fix},
$\gamma$ fixes all points in $B$.
Let $p_i,q_i\in S^1_\infty (\lambda _i)$ be the fixed points of $\gamma$.
Note that $r_i^j\in\{ p_i,q_i\}$ for any $i,j$.
Otherwise,
$B$ cannot be finite
by the same argument as in the proof of Theorem \ref{pc}.
Now, we suppose that $\{r_1^j\mid j\ne 1\}=\{p_1,q_1\}$.
After renumbering the indices if necessary,
we can assume that
$r^j_1=p_1$ for $2\le j<n_1$ and
$r^j_1=q_1$ for $n_1\le j\le n$,
where $3\le n_1\le n$.
Then,
we claim that $r^j_{n_1}=r^1_{n_1}$ for $1\le j<n_1$.
In fact,
let $2\le j<n_1$, and
take $4$ points $x,y,z,w$ as follows:
$x,y$ are in $S^1_\infty(\lambda _1)-\{p_1,q_1\}$
such that the $4$-tuple $(p_1,x,q_1,y)$ 
is circularly ordered,
$z\in S^1_\infty(\lambda _j)$
and $w\in S^1_\infty(\lambda _{n_1})-\{r_{n_1}^1\}$.
Then,
$\sigma _z\in I_{p_1}$, 
$\sigma _w\in I_{q_1}$
and the $4$-tuple
$(I_{p_1},\sigma _x,I_{q_1},\sigma _y)$
is circularly ordered in $S^1_{\univ}$.
Furthermore,
$\sigma _x,\sigma _y\in I_{r_{n_1}^1}$ and
$\sigma _w\notin I_{r_{n_1}^1}$.
It follows that $\sigma _z\in I_{r_{n_1}^1}$,
that is,
$r_{n_1}^1$ is the turning point from $\lambda _j$ to $\lambda _{n_1}$.
This proves the claim.

Now,
if $\{r_{n_1}^j\mid j\ne n_1\}=\{r_{n_1}^1 \}$
we can put $k=n_1$.
Otherwise, by renumbering the indices again,
we can assume that
$r_{n_1}^j=r_{n_1}^1=p_{n_1}$ for $1 \le j<n_2$ $(j\ne n_1)$, and
$r_{n_1}^j=q_{n_1}$ for $n_2\le j\le n$,
where $n_1<n_2\le n$.
Similarly, we have $r_{n_2}^j=r_{n_2}^1$ for $1\le j<n_2$.
Since $B$ is finite,
we can find a desired $k$ 
after repeating this process finitely many times.
\end{proof}

\begin{thm}
\label{z}
Let $M$ and $\mathcal{F}$ be as in Theorem \ref{fix}.
Let $B$ be a finite branch locus of $L$
with nontrivial stabilizer.
Then the stabilizer $\stab (B)$ is isomorphic to $\mathbb{Z}$.
\end{thm}
\begin{proof}
Let $B=\{\lambda _1,\cdots ,\lambda _n\}$,
and let $\gamma$ be as in the proof of the above proposition.
If $\stab (B)$ acts on $B$ trivially,
then for any $\alpha\in\stab (B)$
there exists an integer $k$ such that $\alpha=\gamma ^k$.
Thus $\gamma$ is a generator of $\stab (B)$.
So we assume that $\stab (B)$ acts on $B$ nontrivially.
Let $p_i,q_i\in S^1_\infty(\lambda _i)$ be the fixed points of $\gamma$,
and let $r_i^j\in S^1_\infty(\lambda _i)$ be the turning point
from $\lambda _j$ to $\lambda _i$ for $i\ne j$.
By Proposition \ref{single},
without loss of generality we can assume that
$\{ r^j_1\mid j\ne 1\}$ is a single point.
Put
$\stab (B)(\lambda _1)=\{\alpha (\lambda _1)\mid\alpha \in\stab(B)\}
=\{\lambda _1 ,\cdots ,\lambda _m\}$
where $1<m\le n$.
Since the natural action preserves the set of turning points,
$\{r_i^j\mid 1\le j\le n,j\ne i\}$ is also a single point for any $i\le m$.
Let us denote this single point by $p_i$.
It follows that the subset
$\{\sigma_{p_i}\mid 1\le i\le m\}$ of $S^1_{\univ}$ is kept invariant
by homeomorphisms $\rho _{\univ}(\alpha )$
for $\alpha \in \stab(B)$.
After renumbering indices if necessary, we can assume that
the $m$-tuple
$(\sigma _{p_1},\cdots ,\sigma _{p_m})$ is circularly ordered in $S^1_{\univ}$.
Let $\beta \in\stab(B)$ be such that
$\rho_{\univ}(\beta )(\sigma _{p_1})=\sigma_ {p_2}$, that is,
$\beta (\lambda _1)=\lambda _2$.
Then we have $\beta (\lambda _i)=\lambda _{i+1}$
where the indices $i$ are taken modulo $m$.
Since $\beta\gamma\beta ^{-1}(\lambda_1)=\lambda_1$,
we have $\beta\gamma\beta ^{-1}=\gamma ^k$ for some $k\ne 0$.
Moreover,
there is $l\ne 0$ such that $\beta ^m=\gamma ^l$.
It follows that
$\beta ^{km}=\gamma ^{kl}=\beta\gamma ^l\beta ^{-1}=\beta ^m$,
that is, $\beta ^{(k-1)m}$ is trivial.
If $k\ne 1$,
$\beta$ is a torsion element in $\pi _1 (M)$,
which is a contradiction.
Therefore $k=1$ and we have that
$\gamma$ and $\beta$ commute.
Since $\pi _1 (M)$ is torsion-free,
the subgroup 
$<\gamma,\beta\mid\gamma ^l\beta ^{-m}>$
must be isomorphic to $\mathbb{Z}$.
It follows that there is $\delta\in\pi _1(M)$ such that
$\gamma=\delta ^i$ and $\beta=\delta ^j$ where $i\ne 0$ and $j\ne 0$.
Let $\alpha$ be any element in $\stab(B)$.
Then $\alpha (\lambda_1)=\lambda_i$ for some $1\le i\le m$.
By the choice of $\gamma$ and $\beta$,
we have that
$\alpha$ can be represented as a word in $\gamma,\beta$,
hence, in $\delta$.
It follows that $\stab(B)$ is isomorphic to $\mathbb{Z}$.
\end{proof}

We say that $\alpha\in\pi_1 (M)$ is {\it infinitely divisible}
if for any integer $\ell$,
there are $k>\ell$ and $\beta\in\pi_1 (M)$
such that $\alpha =\beta ^k$. 

\begin{thm}
\label{div}
Let $M$ and $\mathcal{F}$ be as in Theorem \ref{fix},
and let $B$ be a finite branch locus of $L$
such that $\stab (B)$ acts on $B$ nontrivially.
Then 
a generator of $\stab(B)$ $(\cong\mathbb{Z})$ is indivisible.
\end{thm}
\begin{proof}
Let $B=\{\lambda _1,\cdots ,\lambda _n\}$.
By Theorem \ref{z},
$\stab(B)$ is generated by some single element $\alpha$.
We assume by contradiction that $\alpha$ is divisible.
Since $M$ is aspherical (as was noted in the proof of Lemma \ref{fixb}),
$\pi_1(M)$ has no infinitely divisible elements (see \cite[Theorem 4.1]{fr}).
Hence, there exists an indivisible element $\beta\in\pi_1(M)$
such that $\alpha=\beta^k$ for some $k>1$.
Note that 
since $\beta\notin\stab(B)$,
the points
$\lambda _i\in B$ and $\beta (\lambda _i)\in\beta(B)$ are distinct for any $i$.
Moreover, we see that they are incomparable for any $i$.
In fact,
if $\lambda _i$ and $\beta (\lambda _i)$ were comparable,
say $\lambda_i<\beta (\lambda_i)$,
then $\lambda _i <\beta^k (\lambda _i)=\alpha(\lambda_i)$,
contradicting the assumption that $\alpha\in\stab(B)$.
Let $\{\nu _t\}_{0<t<\epsilon}$ be an embedded interval such that
$B=\lim _{t\to 0}\nu_t$.
Since $B$ is $\alpha$-invariant,
it follows from Lemma \ref{invb} that
there is $\nu\in\{\nu _t\}_{0<t<\epsilon}$ such that
$\nu\in C_{\alpha}=C_{\beta^k}$.
We can (and do) take such $\nu$ so that $\nu$ also satisfies that
$\nu\notin C_{\beta}$.
Let $\coprod _{i=1}^l[\hat{\nu}_{i-1},\check{\nu}_{i}]$ $(l>1)$
be the path joining $\nu$ to $\beta (\nu)$.
By the choice of $\nu$ and by Lemma \ref{return},
we have $\beta ^k(\check{\nu}_{m})=\check{\nu}_{m}$
where $m=l/2$.
It follows that $\rho _{\univ}(\beta ^k)$ has a fixed point in $S^1_{\univ}$.
By Lemma \ref{fixb},
$\beta ^k=\alpha$ fixes all points in $B$,
which is a contradiction.
\end{proof}

\begin{rem}
The author does not know whether or not
there is a finite branch locus $B$ such that
$\stab(B)$ acts on $B$ trivially
and is generated by a divisible element.
\end{rem}

We will give an  example of a tautly foliated compact $3$-manifold
admitting a finite branch locus
whose stabilizer acts on the locus nontrivially.
We remark that
a recipe how to construct such a locus has already been 
provided in \cite[Example 3.7]{cd03},
and  our construction follows it.

\begin{ex}
Let  $P=D^2 -(E_1\cup E_2)$ be the unit disk in  $\mathbb{R}^2$
with two open disks removed,
where $E_1,E_2$ are disks centered $(-\frac12,0),(\frac12,0)$
with radius $\frac14$ respectively.
On $P$ we consider a standard singular foliation $\mathcal{G}$
(see Fig. \ref{fig})
satisfying the following properties:
(1) $\mathcal{G}$ has $(0,0)$ as its unique singular point, which is of saddle type,
(2) $\mathcal{G}$ is transverse to $\partial P$,
(3) all leaves of $\mathcal{G}$ (except the $4$ separatrices) are compact,
and
(4) $\mathcal{G}$ is symmetric with respect to both the $x$-axis and $y$-axis.

\begin{figure}[tbp]
\begin{center}
\includegraphics[keepaspectratio, scale=0.5]
{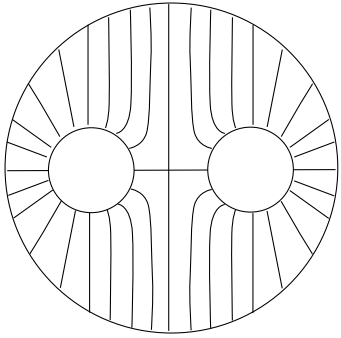}
\caption{A singular foliation $\mathcal{G}$ of $P$.}
\label{fig}
\end{center}
\end{figure}

Put  $S_0=\partial D^2$,
$S_1=\partial E_1$ and $S_2=\partial E_2$.
Let $(P', \mathcal{G}')$ be a copy of $(P,\mathcal{G})$,
and let $c:P'\to P$ be the identification.
We construct a double
$\Sigma=P\cup P'$ using diffeomorphisms
$g_i:S_i\to c(S_i)$  ($i=0,1,2$) to glue $S_i$ to $c(S_i)$,
where
$c^{-1}\circ g_i$ is a rotation by an irrational angle for $i=0$
and by angle $\pi$ for $i=1,2$.
Then $\mathcal{G}$ and $\mathcal{G}'$ 
induce a singular foliation $\mathcal{G}''$ of $\Sigma$
with two saddle singularities
and without any saddle connection.
By construction,
there exists a $\mathcal{G}''$-preserving homeomorphism $\rho$ of $\Sigma$
which fixes the singular points
and which is  the rotation by $\pi$ near the singular points.
Fix a hyperbolic structure on $\Sigma$.
Then each leaf of $\mathcal{G}''$
except the singular points and the separatrices
is isotopic to a unique embedded geodesic,
and the closure of the union of these geodesics constitutes
a geodesic lamination, 
say $\lambda$, on $\Sigma$.
Note that the two complementary regions $Q_1$ and $Q_2$ to $\lambda$
are ideal open squares.
There exists a $\lambda$-preserving 
homeomorphism $\psi$ of $\Sigma$ isotopic to $\rho$.
Let $M$ be the mapping torus of $\psi$, that is,
$M=\Sigma\times [0,1]/(s,1)\sim (\psi (s),0)$.
Then $\lambda$ induces a surface lamination $\Lambda$ of $M$
whose complementary regions $R_i$ are $Q_i$-bundles over $S^1$ for $i=1,2$.
Denote by $p_i:R_i\to S^1$ the bundle projection.
Now we extend $\Lambda$ to a foliation $\mathcal{F}$ of $M$
by filling $R_i$ ($i=1,2$)
with leaves diffeomorphic to $Q_i$ as follows.
Denote the boundary components of $R_i$
by $C_{i1}$ and $C_{i2}$, which are open cylinders.
Let $\gamma_i$ be an oriented loop in $R_i$
such that $p_i\vert \gamma_i$ is a diffeomorphism onto $S^1$.
Then the composition $\gamma_i ^2 =\gamma_i *\gamma_i$
is isotopic to a leaf loop $\gamma_{ij}$ of $C_{ij}$
which is a generator of $\pi _1(C_{ij})$.
We foliate $R_i$ as a product by leaves isotopic to the fibers $Q_i$
so that 
the holonomy along $\gamma_{i1}$ is contracting
and the holonomy along $\gamma_{i2}$ is expanding.
Then the resulting foliation $\mathcal{F}$ is taut
and has two-sided branching,
and each end of a lift of $\gamma_i$ to $\widetilde M$ gives a branch locus
consisting of two points.
Let $\alpha_i$ be an element in $\pi _1(M)$
whose conjugacy class corresponds with the free homotopy class of $\gamma_i$.
Then $\alpha_i$ belongs to the stabilizer of some branch locus
and acts on the locus nontrivially, as desired.
\end{ex}


\section{loops and actions}

Given a loop in a tautly foliated manifold $(M,\mathcal{F})$,
it is natural to ask whether
it is transversable, or tangentiable, to $\mathcal{F}$.
In this section,
we observe that these properties of loops
are expressed completely in the language of the natural action.
Furthermore, we consider relations
between such properties and the branching phenomenon of $\widetilde{\mathcal{F}}$.

We do not need to assume leafwise hyperbolicity in the first two propositions below.

\begin{prop}
\label{basic}
Let $\gamma$ be a loop in $M$,
and $\alpha$ an element in $\pi _1(M,p)$
whose conjugacy class corresponds with the free homotopy class of $\gamma$.
Then, $\gamma$ is tangentiable
if and only if
the action of $\alpha$ on $L$ has a fixed point.
Similarly, $\gamma$ is positively $($resp. negatively$)$ transversable
if and only if
there is a point $\lambda$ in $L$ such that
$\alpha(\lambda)>\lambda$ $($resp. $\alpha(\lambda)<\lambda)$.
\end{prop}
\begin{proof}
Let $\lambda$ be a leaf of $\widetilde{\mathcal{F}}$
and suppose that the deck transformation $\alpha$ leaves $\lambda$ invariant.
Take any point $x$ in $\lambda$
and join $x$ to $\alpha(x)$ by a path in $\lambda$.
Then it projects down to a leaf loop in $M$ freely homotopic to $\alpha$.
Conversely, suppose $\gamma$ is a leaf loop in $M$.
Join the base point $p$ to a point of $\gamma$ by a path $c$.
Then, the loop
$c*\gamma *c^{-1}$
represents an element of $\pi_1(M,p)$ conjugate to $\alpha$.
Obviously it has a fixed point, hence so does $\alpha$.
The claim on transversability is also shown easily.
\end{proof}

According to the proposition,
we can define the tangentiability and transversability
for elements in $\pi_1(M)$ in terms of the natural action:
An element $\alpha$ in $\pi _1(M)$ is {\it tangentiable}
if there is a leaf $\lambda$ in $L$ such that $\alpha(\lambda)=\lambda$.
It is positively (resp. negatively) {\it transversable}
if there is a leaf $\lambda$ in $L$ such that $\alpha(\lambda)>\lambda$
(resp. $\alpha(\lambda)<\lambda$).

We remark here that
$\pi _1(M)$ can have an element
which is neither tangentiable nor transversable.
Such an element exists if and only if
$\mathcal{F}$ has two-sided branching
(Calegari \cite[Lemma 3.2.2]{c00}).

\begin{prop}
\label{pnt}
Let $\alpha\in\pi_1(M)$.
Suppose there are points $\lambda ,\mu\in L$ such that
$\alpha (\lambda)>\lambda$ and $\alpha(\mu) <\mu$.
Then there exists a point $\nu\in L$ such that $\alpha(\nu)=\nu$.
Moreover, if $\lambda$ and $\mu$ are incomparable,
then such $\nu$ can be found in some branch locus.
\end{prop}
\begin{proof}
If $\lambda$ and $\mu$ are comparable,
then the conclusion follows immediately from the intermediate value theorem.
If $\lambda$ and $\mu$ are incomparable,
then the conclusion follows from Lemma \ref{c2}.
\end{proof}

This proposition means that
if a loop in $M$ is both positively and negatively transversable to $\mathcal{F}$,
then it is tangentiable to $\mathcal F$.
Such a statement might also follow from the Poincar{\'e}-Bendixson theorem.
But, here we gave a proof purely in terms of
the natural action of $\pi _1(M)$ on $L$.

In the following we assume leafwise hyperbolicity
and observe that tangentiability and/or transversability
of loops in $M$ and the infiniteness of branch loci
are closely related.

\begin{thm}
\label{infb1}
Let $M$ be a closed oriented $3$-manifold,
and $\mathcal{F}$ a transversely oriented 
leafwise hyperbolic taut foliation of $M$
with one-sided branching.
Suppose that there is non-contractible leaf loop $\gamma$ in $M$
which is not transversable.
Then every branch locus of $L$ is an infinite set.
\end{thm}
\begin{proof}
Suppose that there exists a finite, say positive, branch locus
$B=\{\lambda _1,\cdots ,\lambda _n\}$.
Let $\alpha$ be an element in $\pi _1(M)$
whose conjugacy class corresponds with
the free homotopy class of $\gamma$.
By Proposition \ref{basic},
$\alpha$ has a fixed leaf in $L$,
and for each $\mu\in L$
if $\mu$ is not fixed by $\alpha$ then $\mu\notin C_\alpha$.
Let $\nu$ be a fixed point of $\alpha$.
By Lemma \ref{c1},
for every $\eta$ with $\eta\le\nu$,
we have $\eta\in C_\alpha$,
and therefore $\alpha (\eta)=\eta$.
By replacing $B$ with $\beta(B)$
for some $\beta\in\pi _1(M)$ if necessary,
we can assume that $\lambda _1\le \nu$
and therefore $\alpha(\lambda_1)=\lambda_1$.
This implies in particular that $B$ is $\alpha$-invariant.
Since $B$ is finite,
by Theorem \ref{fix},
we have $\alpha(\lambda _i)=\lambda_i$ for any $1\le i\le n$.
By Lemma \ref{key}, $\alpha$ must be trivial,
which is a contradiction.
\end{proof}

\begin{cor}
Let $M$ be a closed oriented $3$-manifold,
and $\mathcal{F}$ a transversely oriented 
leafwise hyperbolic taut foliation of $M$
with every leaf dense.
Suppose that there is non-contractible leaf loop $\gamma$ in $M$
which is not transversable.
Then every branch locus of $L$ is an infinite set.
\end{cor}
\begin{proof}
Suppose there is a finite branch locus $B$.
Let $\alpha\in\pi _1(M)$ be as in the proof of the preceding theorem.
By Proposition \ref{co}, there is an embedded open interval $I\subset L$
such that $I$ is contained in $C_\alpha$.
Since every leaf of $\mathcal{F}$ is dense,
there is $\beta\in\pi _1(M)$ such that
$\beta(B)\cap I\ne\emptyset$.
Then the same argument as in Theorem \ref{infb1} shows the conclusion.
\end{proof}

\begin{thm}
\label{infb2}
Let $M$ be a closed oriented $3$-manifold,
and $\mathcal{F}$ a transversely oriented
leafwise hyperbolic taut foliation of $M$ with branching.
Suppose that there is a non-contractible leaf loop $\gamma$ in $M$
which is not transversable.
Then $L$ has an infinite branch locus.
\end{thm}
\begin{proof}
Let $\alpha$ be as in Theorem \ref{infb1}.
Then $\alpha$ has a
fixed point $\nu\in L$,
and for each $\mu\in L$
if $\mu$ is not fixed by $\alpha$ then $\mu\notin C_\alpha$.
Without loss of generality, we assume that
$\mathcal{F}$ has a positive branch locus.
We claim that there exist some $\nu '>\nu$ such that
$\nu '$ and $\alpha (\nu ')$ are incomparable.
Put $L'=\{\mu\mid\mu >\nu\}$.
Then we can observe that $L'$ is a submanifold of $L$
with one-sided branching in the positive direction
and contains at least one branch locus.
If $\alpha$ fixes all leaves in $L'$,
then by applying Lemma \ref{key} to a branch locus in $L'$
we obtain that $\alpha$ is trivial in $\pi _1(M)$,
which contradicts the hypothesis that
$\alpha$ is represented by a non-contractible loop.
Therefore, there exists some $\nu '\in L$
which is not fixed by $\alpha$.
Since such $\nu '$ does not belong to $C_\alpha$,
the claim is shown.
Since $\nu <\nu '$ and $\alpha (\nu)=\nu$,
it follows that
$\nu$ is a common lower bound for $\nu '$ and $\alpha (\nu ')$.
Thus,
the fact that $\nu '$ and $\alpha (\nu ')$ are incomparable implies that
there is a unique $\lambda\in (\nu ,\nu' ]$ such that
$\mu\in [\nu ,\nu ']$ is fixed by $\alpha$
if and only if $\mu\in [\nu ,\lambda)$.
Evidently, $\lambda$ belongs to some $\alpha$-invariant branch locus,
say $B$.
Also note that $\rho _{\univ}(\alpha)$ has a fixed point
because $\alpha$ fixes a point in $L$.
We now show $B$ is infinite.
Suppose not.
Then, by Lemma \ref{fixb},
all leaves in $B$ are $\alpha$-fixed,
and, by Lemma \ref{key},
$\alpha$ must be trivial,
which is a contradiction.
\end{proof}

\end{document}

%% file: rfig.tex
\unitlength 0.1in
\begin{picture}( 43.2000, 27.5000)(  6.1000,-28.5000)
%
{\color[named]{Black}{%
\special{pn 13}%
\special{pa 3900 700}%
\special{pa 3900 100}%
\special{fp}%
}}%
%
{\color[named]{Black}{%
\special{pn 13}%
\special{pa 1700 700}%
\special{pa 1700 100}%
\special{fp}%
}}%
%
{\color[named]{Black}{%
\special{pn 13}%
\special{pa 2800 2000}%
\special{pa 2800 2800}%
\special{fp}%
}}%
%
{\color[named]{Black}{%
\special{pn 13}%
\special{pa 4800 2800}%
\special{pa 4000 800}%
\special{fp}%
}}%
%
{\color[named]{Black}{%
\special{pn 13}%
\special{pa 3800 800}%
\special{pa 2900 1900}%
\special{fp}%
}}%
%
{\color[named]{Black}{%
\special{pn 13}%
\special{pa 2700 1900}%
\special{pa 1800 800}%
\special{fp}%
}}%
%
{\color[named]{Black}{%
\special{pn 13}%
\special{pa 800 2810}%
\special{pa 1600 810}%
\special{fp}%
}}%
%
{\color[named]{Black}{%
\special{pn 0}%
\special{sh 1.000}%
\special{ia 800 2800 50 50  0.0000000  6.2831853}%
}}%
{\color[named]{Black}{%
\special{pn 8}%
\special{pn 8}%
\special{ar 800 2800 50 50  0.0000000  6.2831853}%
}}%
%
{\color[named]{Black}{%
\special{pn 0}%
\special{sh 1.000}%
\special{ia 800 2800 50 50  0.0000000  6.2831853}%
}}%
{\color[named]{Black}{%
\special{pn 8}%
\special{pn 8}%
\special{ar 800 2800 50 50  0.0000000  6.2831853}%
}}%
%
{\color[named]{Black}{%
\special{pn 0}%
\special{sh 1.000}%
\special{ia 800 2800 50 50  0.0000000  6.2831853}%
}}%
{\color[named]{Black}{%
\special{pn 8}%
\special{pn 8}%
\special{ar 800 2800 50 50  0.0000000  6.2831853}%
}}%
%
{\color[named]{Black}{%
\special{pn 0}%
\special{sh 1.000}%
\special{ia 800 2800 50 50  0.0000000  6.2831853}%
}}%
{\color[named]{Black}{%
\special{pn 8}%
\special{pn 8}%
\special{ar 800 2800 50 50  0.0000000  6.2831853}%
}}%
%
{\color[named]{Black}{%
\special{pn 0}%
\special{sh 1.000}%
\special{ia 4800 2800 50 50  0.0000000  6.2831853}%
}}%
{\color[named]{Black}{%
\special{pn 8}%
\special{pn 8}%
\special{ar 4800 2800 50 50  0.0000000  6.2831853}%
}}%
%
{\color[named]{Black}{%
\special{pn 0}%
\special{sh 1.000}%
\special{ia 1600 800 50 50  0.0000000  6.2831853}%
}}%
{\color[named]{Black}{%
\special{pn 8}%
\special{pn 8}%
\special{ar 1600 800 50 50  0.0000000  6.2831853}%
}}%
%
{\color[named]{White}{%
\special{pn 0}%
\special{sh 1}%
\special{ia 2800 2000 50 50  0.0000000  6.2831853}%
}}%
{\color[named]{Black}{%
\special{pn 8}%
\special{pn 8}%
\special{ar 2800 2000 50 50  0.0000000  6.2831853}%
}}%
%
{\color[named]{White}{%
\special{pn 0}%
\special{sh 1}%
\special{ia 1700 700 50 50  0.0000000  6.2831853}%
}}%
{\color[named]{Black}{%
\special{pn 8}%
\special{pn 8}%
\special{ar 1700 700 50 50  0.0000000  6.2831853}%
}}%
%
{\color[named]{White}{%
\special{pn 0}%
\special{sh 1}%
\special{ia 3900 700 50 50  0.0000000  6.2831853}%
}}%
{\color[named]{Black}{%
\special{pn 8}%
\special{pn 8}%
\special{ar 3900 700 50 50  0.0000000  6.2831853}%
}}%
%
{\color[named]{Black}{%
\special{pn 0}%
\special{sh 1.000}%
\special{ia 1800 800 50 50  0.0000000  6.2831853}%
}}%
{\color[named]{Black}{%
\special{pn 8}%
\special{pn 8}%
\special{ar 1800 800 50 50  0.0000000  6.2831853}%
}}%
%
{\color[named]{Black}{%
\special{pn 0}%
\special{sh 1.000}%
\special{ia 3800 800 50 50  0.0000000  6.2831853}%
}}%
{\color[named]{Black}{%
\special{pn 8}%
\special{pn 8}%
\special{ar 3800 800 50 50  0.0000000  6.2831853}%
}}%
%
{\color[named]{Black}{%
\special{pn 0}%
\special{sh 1.000}%
\special{ia 4000 800 50 50  0.0000000  6.2831853}%
}}%
{\color[named]{Black}{%
\special{pn 8}%
\special{pn 8}%
\special{ar 4000 800 50 50  0.0000000  6.2831853}%
}}%
%
{\color[named]{Black}{%
\special{pn 0}%
\special{sh 1.000}%
\special{ia 2700 1900 50 50  0.0000000  6.2831853}%
}}%
{\color[named]{Black}{%
\special{pn 8}%
\special{pn 8}%
\special{ar 2700 1900 50 50  0.0000000  6.2831853}%
}}%
%
{\color[named]{Black}{%
\special{pn 0}%
\special{sh 1.000}%
\special{ia 2900 1900 50 50  0.0000000  6.2831853}%
}}%
{\color[named]{Black}{%
\special{pn 8}%
\special{pn 8}%
\special{ar 2900 1900 50 50  0.0000000  6.2831853}%
}}%
\put(26.5000,-19.1000){\makebox(0,0)[rt]{{\large $\check{\nu}_m$}}}%
\put(30.4000,-19.1000){\makebox(0,0)[lt]{{\large $\hat{\nu}_m$}}}%
\put(6.1000,-28.0000){\makebox(0,0)[lt]{$\lambda$}}%
\put(49.0000,-28.0000){\makebox(0,0)[lt]{$\alpha(\lambda)$}}%
%
{\color[named]{Black}{%
\special{pn 8}%
\special{pa 4700 2850}%
\special{pa 3900 850}%
\special{da 0.070}%
}}%
%
{\color[named]{Black}{%
\special{pn 8}%
\special{pa 3890 860}%
\special{pa 2900 2080}%
\special{da 0.070}%
}}%
%
{\color[named]{Black}{%
\special{pn 8}%
\special{pa 2900 2090}%
\special{pa 2900 2690}%
\special{da 0.070}%
}}%
\put(29.5000,-26.8000){\makebox(0,0)[lt]{{\large $\alpha(\check{\nu}_m)$}}}%
%
{\color[named]{Black}{%
\special{pn 0}%
\special{sh 1.000}%
\special{ia 2900 2680 30 30  0.0000000  6.2831853}%
}}%
{\color[named]{Black}{%
\special{pn 8}%
\special{pn 8}%
\special{ar 2900 2680 30 30  0.0000000  6.2831853}%
}}%
%
{\color[named]{Black}{%
\special{pn 13}%
\special{pa 4930 2810}%
\special{pa 4130 810}%
\special{dt 0.045}%
}}%
%
{\color[named]{Black}{%
\special{pn 13}%
\special{pa 3670 800}%
\special{pa 3150 1450}%
\special{dt 0.045}%
}}%
%
{\color[named]{Black}{%
\special{pn 0}%
\special{sh 1.000}%
\special{ia 3160 1440 30 30  0.0000000  6.2831853}%
}}%
{\color[named]{Black}{%
\special{pn 8}%
\special{pn 8}%
\special{ar 3160 1440 30 30  0.0000000  6.2831853}%
}}%
\put(31.0000,-13.7000){\makebox(0,0)[rt]{{\large $\alpha(\check{\nu}_m)$}}}%
\put(46.4000,-17.1000){\makebox(0,0)[lt]{{\large $\alpha(\delta_0)$}}}%
\put(40.0000,-22.2000){\makebox(0,0)[lt]{{\large $\alpha(\delta_0)$}}}%
\end{picture}%